\documentclass{article}
\usepackage[utf8]{inputenc}
\usepackage{amsmath,amsthm,dsfont,pifont,amsfonts,MnSymbol}
\usepackage{verbatim}

\title{Every $T_1$ connected first-countable space is a continuous open image of a connected metrizable space  
\footnote{Keywords\textup{:} open map, connected space, metrizable space, first-countable space}}
\author{Vlad Smolin \\ \small{Krasovskii Institute of Mathematics and Mechanics, Ural Federal University,} \\ \small{ Yekaterinburg, Russia} \\ \small{e-mail: SVRusl@yandex.ru}}

\theoremstyle{plain}
\newtheorem{teo}{Theorem}
\newtheorem{lemm}[teo]{Lemma}
\newtheorem{corr}[teo]{Corollary}

\newtheorem{ques}[teo]{Question}

\theoremstyle{definition}
\newtheorem{deff}[teo]{Definition}

\newtheorem{nota}[teo]{Notation}

\renewcommand{\phi}{\varphi}

\begin{document}

\maketitle
\begin{abstract}
Answering a question posed by Vladimir Tkachuk, we prove that every $T_1$ connected first-countable space is a continuous open image of a connected metrizable space.
\end{abstract}

\section{Introduction}

In \cite{Tkach} author posed three questions:

\begin{itemize}
    \item[\ding{43}\,] Let $X$ be a Tychonoff sequential connected space. Is it true that $X$ is a continuous quotient image of a metrizable connected space?
    \item[\ding{43}\,] Let $X$ be a Tychonoff Fr\'echet connected space. Is it true that $X$ is a continuous hereditarily quotient image of a metrizable connected space?
    \item[\ding{43}\,] Let $X$ be a Tychonoff first-countable connected space. Is it true that $X$ is a continuous open image of a metrizable connected space?
\end{itemize}

These problems were motivated by the following classical results in general topology(for details, see \cite[4.2.D]{eng}):
\begin{itemize}
    \item[\ding{164}\,] Every Hausdorff sequential space is a continuous quotient image of a metrizable space.
    \item[\ding{164}\,] Every Hausdorff Fr\'echet space is a continuous hereditarily quotient image of a metrizable space.
    \item[\ding{164}\,] Every first-countable space is a continuous open image of a metrizable space.
\end{itemize}

It is worth noting that in each of these results, the constructed metrizable space is always disconnected, so it cannot be used as a solution. In \cite{Fed} authors gave a positive answer to the first question, and a positive answer to the second question followed.
However, the third question remained unanswered.
In this paper we give a positive answer to the third question. The main idea of our proof is the following: if we are given an open continuous map $g$ from a specific metrizable $M$ space onto a $T_1$ connected first-countable space $Z$, then we can find a connected metrizable topology $\gamma \subseteq \tau_M$ such that $g\colon \langle M, \gamma\rangle \to Z$ is continuous.

\section{The proof}
We use terminology from \cite{enc}.
\begin{nota} .
    \begin{itemize}
        \item [\ding{46}\,] If $g$ and $f$ are functions, then ${g}\circ {f}$ is the composition of ${g}$ and ${f}$ (that is, ${g}$ after ${f}$);
        \item [\ding{46}\,] if $f$ is a function and $A$ is a subset of its domain, then $f[A] \coloneq \{f(a)\ \colon\ a \in A\}$;
        \item[\ding{46}\,] if $A$ is a set, then $\mathsf{D}(A)$ is the discrete topological space with the underlying set $A$.
    \end{itemize}
\end{nota}
The following is a key tool for our construction.

\begin{deff}\cite{Roit}
    Suppose that $X$ and $Y$ are topological spaces and let $\phi\colon Y \to X$ be a function (not necessarily continuous). Then
    $$
        \sigma(\tau_X, \tau_Y, \phi) \coloneq \{U \in \tau_X\ \colon\ \phi^{-1}[U] \in \tau_Y\}.
    $$
In \cite{Roit} authors prove that $\sigma(\tau_X, \tau_Y, \phi)$ is a topology on the set $X$.
\end{deff}

\begin{lemm} \label{main_lemma}
    Let $X$ and $Y$ be topological spaces, $D_X$ and $D_Y$ closed discrete subsets of $X$ and $Y$ respectively, and $\phi\colon Y \to X$ an injection such that $\phi[Y\setminus D_Y] = D_X$ (or equivalently $\phi[D_Y] = \phi[Y] \setminus D_X$) and $\phi[Y]$ is dense in $X$. Then
    \begin{itemize}
        \item[1.] If $D$ is a closed discrete subset of $X$ and $\phi^{-1}[D]$ is $F_\sigma$ in $Y$, then $D$ is $F_\sigma$ in $\langle X, \sigma(\tau_X, \tau_Y, \phi) \rangle$.
        \item[2.] If $D$ is a closed discrete subset of $X$ and $\phi^{-1}[D] = \bigcup_{\alpha \in A} F_\alpha$, where $\{F_\alpha\ \colon\ \alpha \in A\}$ is a discrete family of closed sets in $Y$, then $\{\phi[F_\alpha]\ \colon\ \alpha \in A\}$ is a discrete family of closed sets in $\langle X, \sigma(\tau_X, \tau_Y, \phi) \rangle$.
        \item[3.] If $D$ is a closed discrete subset of $X$ and $\phi^{-1}[D] = \bigcup_{\alpha \in A} F_\alpha$, where $\{F_\alpha\ \colon\ \alpha \in A\}$ is a $\sigma$-discrete family of closed sets in $Y$, then $\{\phi[F_\alpha]\ \colon\ \alpha \in A\}$ is a $\sigma$-discrete family of closed sets in $\langle X, \sigma(\tau_X, \tau_Y, \phi) \rangle$.
        \item[4.] If $D$ is a closed discrete subset of $X$ and $\phi^{-1}[D]$ is a closed discrete subset of $Y$, then $D$ is a closed discrete subset of $\langle X, \sigma(\tau_X, \tau_Y, \phi) \rangle$.
        \item[5.] If every open subset of $X$ is $F_\sigma$ and every open subset of $Y$ is $F_\sigma$, then every open subset of $\langle X, \sigma(\tau_X, \tau_Y, \phi) \rangle$ is $F_\sigma$.
        \item[6.] If $X$ and $Y$ are collectionwise normal, then $\langle X, \sigma(\tau_X, \tau_Y, \phi) \rangle$ is collectionwise normal.
        \item[7.] If $X$ and $Y$ are metrizable, then $\langle X, \sigma(\tau_X, \tau_Y, \phi) \rangle$ is submetrizable.
    \end{itemize}
\end{lemm}

\begin{proof}
Note that $F$ is a closed subset of $\langle X, \sigma(\tau_X, \tau_Y, \phi) \rangle$ if and only if $F$ is closed in $X$ and $\phi^{-1}[F]$ is closed in $Y$.

    (1) Let $D_2 \coloneq \bigcup \{F \subseteq D\ \colon\ \phi^{-1}[F] = \varnothing\}$. Since $D$ is a closed discrete subset of $X$, we see that $D_2$ is closed in $X$, and so $D_2$ is a closed subset of $\langle X, \sigma(\tau_X, \tau_Y, \phi) \rangle$, since $\phi^{-1}[D_2] = \varnothing$ is closed in $Y$.
    
    Let $\phi^{-1}[D] = \bigcup_{i \in \omega} F_i$, where $F_i$ is closed in $Y$ for all $i \in \omega$. Since $D$ is a closed discrete set, then $\phi[F_i]$ is closed in $X$ for all $i \in \omega$, also $\phi^{-1}[\phi[F_i]] = F_i$ for all $i \in \omega$, and so we see that $\phi[F_i]$ is closed in $\langle X, \sigma(\tau_X, \tau_Y, \phi) \rangle$ for all $i \in \omega$. And so $D = D_2 \cup \bigcup_{i \in \omega} \phi[F_i]$ is $F_\sigma$ in $\langle X, \sigma(\tau_X, \tau_Y, \phi) \rangle$.

    (2) Let $A^\prime \subseteq A$, we show that $\bigcup_{\alpha \in A^\prime}\phi[F_\alpha]$ is closed in $\langle X, \sigma(\tau_X, \tau_Y, \phi) \rangle$. Since $D$ is a closed discrete set, we see that $\bigcup_{\alpha \in A^\prime}\phi[F_\alpha]$ is closed in $X$, also $\phi^{-1}[\bigcup_{\alpha \in A^\prime}\phi[F_\alpha]] = \bigcup_{\alpha \in A^\prime}F_\alpha$ is closed in $Y$, and so $\bigcup_{\alpha \in A^\prime}\phi[F_\alpha]$ is closed in $\langle X, \sigma(\tau_X, \tau_Y, \phi) \rangle$.

    (3) Let $\{F_\alpha\ \colon\ \alpha \in A\} = \bigcup_{i \in \omega} \mathcal{F}_i$, where $\mathcal{F}_i$ is a discrete family of closed sets in $Y$ for all $i \in \omega$. Then apply (2) to $\mathcal{F}_i$ for every $i \in \omega$.

    (4) Let $D_2 \coloneq \bigcup \{F \subseteq D\ \colon\ \phi^{-1}[F] = \varnothing\}$ and $D_1 \coloneq D \setminus D_2$. Every subset $F$ of $D_2$ is closed in $X$ and $\phi^{-1}[F] = \varnothing$ is closed in $Y$, therefore $F$ is closed in $\langle X, \sigma(\tau_X, \tau_Y, \phi) \rangle$. So we see that $D_2$ is a closed discrete subset of $\langle X, \sigma(\tau_X, \tau_Y, \phi) \rangle$. Every subset $F$ of $D_1$ is closed in $X$ and $\phi^{-1}[F]$ is closed in $Y$, and so it is closed in $\langle X, \sigma(\tau_X, \tau_Y, \phi) \rangle$. Now we see that $D$ is a closed discrete subset of $\langle X, \sigma(\tau_X, \tau_Y, \phi) \rangle$.

    (5) Let $U \in \sigma(\tau_X, \tau_Y, \phi)$, then $U \in \tau_X$ and $\phi^{-1}[U] \in \tau_Y$. Note that $U = (D_X \cap U) \cup (U \setminus D_X)$. Then $\phi^{-1}[U] = \phi^{-1}[D_X \cap U] \cup \phi^{-1}[U \setminus D_X]$.

    Let $U \setminus D_X = \bigcup_{i \in \omega} F_i$, where $F_i$ is closed in $X$ for all $i \in \omega$. Since $\phi^{-1}[U \setminus D_X] \subseteq D_Y$, we see that $\phi^{-1}[F_i]$ is closed in $Y$ for all $i \in \omega$, and so $U \setminus D_X$ is $F_\sigma$ in $\langle X, \sigma(\tau_X, \tau_Y, \phi) \rangle$.

    Since $\phi^{-1}[U] \in \tau_Y$ and $\phi^{-1}[U \setminus D_X] \subseteq D_Y$, then $$\phi^{-1}[U \cap D_X] = \phi^{-1}[U \setminus (U \setminus D_X)] = \phi^{-1}[U] \setminus \phi^{-1}[U \setminus D_X] \in \tau_Y.$$ Then from (1) it follows that $\phi^{-1}[U \cap D_X]$ is $F_\sigma$ in $\langle X, \sigma(\tau_X, \tau_Y, \phi) \rangle$.

    (6) The proof is the same as the proof of Theorem 2.2 in \cite{Roit}. Replace $C$ with $D_X$, $\mathsf{J}(\kappa)$ with $Y$, $T$ with $D_Y$ and $\phi$ with $\phi$.

    (7) Let $\rho_X$ be a metric on $X$ such that $\rho_X$ generates $\tau_X$ and $\rho_X(d_1, d_2) \geq 1$ for all distinct $d_1, d_2 \in D_X$. It follows from \cite[Lemma 2.4]{Roit} that such metric exists. Let $\rho_Y$ be a metric on $Y$ such that $\rho_Y$ generates $\tau_Y$. Now we use the construction from \cite{Roit}. For $x, y \in X$ let
        \[
            \lambda_0(x, y)\coloneq\begin{cases}
            \mathsf{min}\{\rho_X(x, y), \rho_Y(\phi^{-1}(x), \phi^{-1}(y))\},&\text{if $\{x, y\}\subseteq\phi[Y]$;}\\
            \rho_X(x, y),&\text{otherwise}.
            \end{cases}
        \]
    Now for $x, y \in X$ let 
    $$
        \lambda(x, y) \coloneq \mathsf{inf}\{\lambda_0(x_0, x_1) + \dots + \lambda_0(x_{n-1}, x_n)\},
    $$
    where $x_0, \dots, x_n$ varies over all finite sequences in $X$ with $x_0=x$ and $x_n=x$. Applying the same reasoning as in the proof of Theorem 2.5 in \cite{Roit}, we can show that $\lambda$ is a metric. It is easy to see that $\lambda(x, y) \leq \rho_X(x,y)$ for all $x, y \in X$, and so $\tau_\lambda \subseteq \tau_X$ ($\tau_\lambda$ is the topology on $X$ generated by $\lambda$). Now we prove that $\tau_\lambda \subseteq \sigma(\tau_X, \tau_Y, \phi)$.

    Fix $x \in X$ and $\varepsilon > 0$, we need to prove that $\phi^{-1}[O^{\lambda}_\varepsilon(x)] \in \tau_Y$. Take an arbitrary $y \in \phi^{-1}[O^{\lambda}_\varepsilon(x)]$. Now let $x_1, \dots, x_n$ be such that $\lambda_0(x, x_1) + \lambda_0(x_1, x_2) + \dots + \lambda_0(x_{n-1}, x_n) + \lambda_0(x_n, \phi(y)) < \varepsilon$, then there exists $\delta > 0$ such that for all $z \in O^{\rho_Y}_{\delta}(y)$ we have $\lambda_0(x, x_1) + \lambda_0(x_1, x_2) + \dots + \lambda_0(x_{n-1}, x_n) + \lambda_0(x_n, \phi(y)) + \rho_Y(y, z) < \varepsilon$, and so $\lambda(x, \phi(z)) < \varepsilon$ for all $z \in O^{\rho_Y}_{\delta}(y)$, but it means that $O^{\rho_Y}_{\delta}(y) \subseteq \phi^{-1}[O^{\lambda}_\varepsilon(x)]$.
\end{proof}

\begin{teo} \label{main_theorem}
    Every $T_1$ connected first-countable space is a continuous open image of a connected metrizable space.
\end{teo}

\begin{proof}
Let $Z$ be a connected first-countable space and let $Q$ be a dense subset of $Z$. From \cite[Theorem 1]{Pol} it follows that there exist a $\sigma$-discrete metrizable space $T$ and a continuous open map $f\colon T \to Z$. Let $\kappa \coloneq \mathsf{max}\{2^{\aleph_0}, |T|\}$. Consider the space $\mathsf{D}(\kappa)\times T$ and let
$$
    g \coloneq (f \circ p_T) \colon \mathsf{D}(\kappa)\times T \to Z,
$$
where $p_T$ is the projection of $\mathsf{D}(\kappa)\times T$ onto $T$. Note that
\begin{itemize}
    \item $g$ is open and continuous;
    \item $|g^{-1}(z)|=\kappa$ for all $z \in Z$;
    \item $g^{-1}[Q]$ is dense in $\mathsf{D}(\kappa)\times T$.
\end{itemize}

Now let us prove that there exists a closed discrete subset $D$ of $\mathsf{D}(\kappa)\times T$ such that
\begin{itemize}
    \item $D \subseteq g^{-1}[Q]$;
    \item $|g^{-1}(q) \cap D| = \kappa$ for all $q \in Q$;
    \item $|g^{-1}(q) \setminus D| = \kappa$ for all $q \in Q$.
\end{itemize}

Note that $|Q| \leq \kappa$. Let $\{A_q\ \colon\ q \in Q\}$ be a decomposition of $\kappa$ such that $|A_q| = \kappa$ for all $q \in Q$. Now for every $q \in Q$ and $\alpha \in A_q$ fix an element $\mathsf{x}(\alpha, q) \in g^{-1}(q)\cap(\{\alpha\}\times T)$. Let $D \coloneq \{\mathsf{x}(\alpha, q)\ \colon\ q \in Q \text{ and } \alpha \in A_q\}$.

Let $\mathsf{J}(\kappa)$ be the hedgehog of spines $\kappa$ (for details see \cite[b-13 Special Spaces]{enc}). For every $q \in Q$ let $\phi_{q} \colon \mathsf{J}(\kappa) \to g^{-1}(q)$ be a bijection such that
$$
    \phi_{q}[\{1\}\times \kappa] = g^{-1}(q)\setminus D.
$$
Now consider the space $\mathsf{D}(Q) \times \mathsf{J}(\kappa)$ and define $\phi \colon \mathsf{D}(Q) \times \mathsf{J}(\kappa) \to \mathsf{D}(\kappa)\times T$ by the following rule:
$$
    \phi(\langle q, x\rangle) \coloneq \phi_{q}(x).
$$
Note that
\begin{equation} \label{as_in_lem1}
    \phi[\mathsf{D}(Q) \times \mathsf{J}(\kappa)] = g^{-1}[Q] \text{ is dense in } \mathsf{D}(\kappa)\times T,
\end{equation}
and
\begin{equation} \label{as_in_lem2}
    \phi[\mathsf{D}(Q) \times (\{1\} \times \kappa)] = g^{-1}[Q] \setminus D.
\end{equation}
Now let 
$$X \coloneq \mathsf{D}(\kappa)\times T,$$ $$Y \coloneq \mathsf{D}(Q) \times \mathsf{J}(\kappa),$$ $$D_X \coloneq D \text{ and }$$ $$D_Y \coloneq \mathsf{D}(Q) \times (\{1\} \times \kappa).$$
From (\ref{as_in_lem1}) and (\ref{as_in_lem2}) it follows that $X$, $Y$, $D_X$, $D_Y$, and $\phi$ satisfy the premises of Lemma \ref{main_lemma}. Let $\sigma \coloneq \sigma(\tau_X, \tau_Y, \phi)$. Now we prove that
\begin{equation} \label{sigma_connected}
    \langle X, \sigma \rangle \text{ is connected.}
\end{equation}

Assume the converse, let $U_1$ and $U_2$ be nonempty open subsets of $\langle X, \sigma \rangle$ such that $U_1 \cap U_2 = \varnothing$ and $U_1 \cup U_2 = X$. Note that $g[U_1] \cap g[U_2] \neq \varnothing$, since $g\colon X \to Z$ is an open surjection, $\sigma \subseteq \tau_X$, and $Z$ is connected. Take $q \in g[U_1] \cap g[U_2] \cap Q$, then $g^{-1}(q)\cap U_1 \neq \varnothing$ and $g^{-1}(q)\cap U_2 \neq \varnothing$. It follows that $\phi^{-1}[U_1] \cap \phi^{-1}(g^{-1}(q))$ and $\phi^{-1}[U_2] \cap \phi^{-1}(g^{-1}(q))$ are nonempty disjoint open subsets of $\phi^{-1}(g^{-1}(q)) \approx \mathsf{J}(\kappa)$ that cover $\phi^{-1}(g^{-1}(q))$, a contradiction.

Let us prove that there exists a $\sigma$-discrete family of closed sets $\mathcal{D}$ in $\langle X, \sigma \rangle$ such that 
\begin{equation} \label{DX_D}
    D_X = \bigcup \mathcal{D}
\end{equation} 
and
\begin{equation} \label{F_sing}
    \forall F \in \mathcal{D}\ \exists q \in Q [F \subseteq g^{-1}(q)].
\end{equation}
Since $\mathsf{D}(Q) \times \mathsf{J}(\kappa)$ is metrizable, we see that $\{q\} \times [\mathsf{J}(\kappa) \setminus (\{1\} \times \kappa)]$ is an $F_\sigma$-set for all $q \in Q$, so for every $q \in Q$ let $\{F_q(i)\ \colon\ i \in \omega\}$ be a family of closed sets such that $\{q\} \times [\mathsf{J}(\kappa) \setminus (\{1\} \times \kappa)] = \bigcup_{i \in \omega} F_q(i)$. For every $i \in \omega$ let $\mathcal{F}_i \coloneq \{F_q(i)\ \colon\ q \in Q\}$, it is easy to see that $\mathcal{F}_i$ is a discrete family of closed sets in $\mathsf{D}(Q)\times \mathsf{J}(\kappa)$. Note that 
$$
    \forall q \in Q\ \forall i \in \omega \big[\phi[F_q(i)] = \phi_q[F_q(i)] \subseteq g^{-1}(q) \cap D_X\big]; 
$$
$$
    D_X = \bigcup \bigcup_{i \in \omega} \phi[\mathcal{F}_i].
$$
Now let $\mathcal{D} \coloneq \bigcup_{i\in\omega}\phi[\mathcal{F}_i]$. The statement follows from Lemma \ref{main_lemma}.3.

Now we prove that 
\begin{equation} \label{X_minus_DX_sigma_disc}
X \setminus D_X \text{ is } \sigma\text{-discrete in }\langle X, \sigma \rangle.
\end{equation}
Note that $X \setminus D_X$ is $\sigma$-discrete in $\langle X, \tau_X \rangle$ and $\phi^{-1}[X \setminus D_X] = \mathsf{D}(Q) \times (\{1\}\times \kappa)$ is a closed discrete in $Y$. Now the statement follows from Lemma \ref{main_lemma}.4.

From Lemma \ref{main_lemma}.5 and Lemma \ref{main_lemma}.6 it follows that
\begin{equation} \label{sigma_colnorm_perfnorm}
    \langle X, \sigma \rangle \text{ is collectionwise normal and perfectly normal.} 
\end{equation}

Also from Lemma \ref{main_lemma}.7 it follows that
\begin{equation} \label{sigma_submetr}
    \langle X, \sigma \rangle \text{ is submetrizable.} 
\end{equation}

Now let us prove that
$$
    \forall U \in \tau_Z\big[g^{-1}[U]\in\sigma\big].
$$

Let $U \in \tau_Z$, then $\phi^{-1}[g^{-1}[U]] = \phi^{-1}[g^{-1}[U] \cap g^{-1}[Q]] = \phi^{-1}[g^{-1}[U\cap Q]] = \bigcup\{ \{q\}\times \mathsf{J}(\kappa)\ \colon\ q \in U \cap Q\} \in \tau_Y$.

Now we prove that there exists a metrizable topology $\gamma$ on $X$ such that $\gamma \subseteq \sigma$ and 
\begin{equation*}
    \forall U \in \tau_Z \big[g^{-1}[U]\in \gamma\big]. 
\end{equation*}

For every $z \in Z$ let $\mathcal{U}_z = \{U_z^j\ \colon\ j \in \omega\}$ be a countable local base at $z$. From (\ref{X_minus_DX_sigma_disc}) it follows that 
\begin{equation} \label{X_minus_DX_Di}
    X \setminus D_X = \bigcup_{i\in\omega} D_i,
\end{equation}
where $D_i$ is a closed discrete set in $\langle X, \sigma \rangle$ for all $i \in \omega$. From (\ref{sigma_colnorm_perfnorm}) it follows that for every $i \in \omega$ there exists a discrete family of open sets $\mathcal{W}_i = \{W_i(d)\ \colon\ d\in D_i\}$ in $\langle X, \sigma \rangle$ such that $d \in W_i(d)$ for all $d \in D_i$.

For every $i, j \in \omega$ let
$$
    \mathcal{W}_i^j \coloneq \{W_i(d) \cap g^{-1}[U_{g(d)}^j]\ \colon\ d \in D_i\}.
$$

From (\ref{DX_D}) it follows that 
$$\mathcal{D} = \bigcup_{i \in \omega}\mathcal{D}_i,$$
where $\mathcal{D}_i$ is a discrete family of closed sets in $\langle X, \sigma \rangle$ for all $i \in \omega$. From (\ref{sigma_colnorm_perfnorm}) it follows that for every $i \in \omega$ there exists a discrete family of open sets $\mathcal{V}_i = \{V_i(F)\ \colon\ F \in \mathcal{D}_i\}$ in $\langle X, \sigma \rangle$ such that $F \subseteq V_i(F)$ for all $F \in \mathcal{D}_i$. Also note that from (\ref{F_sing}) it follows that $g[F]$ is a singleton for all $F \in \mathcal{D}$. Now for every $i, j \in \omega$ let 
$$
    \mathcal{V}_i^j \coloneq \{V_i(F) \cap g^{-1}[U_{g[F]}^j]\ \colon\ F \in \mathcal{D}_i\}.
$$

From (\ref{sigma_submetr}) and \cite[Lemma 3.1]{oka} it follows that there exists a metrizable topology $\gamma$ on $X$ such that $\gamma \subseteq \sigma$, $\mathcal{W}_i^j \subseteq \gamma$, and $\mathcal{V}_i^j \subseteq \gamma$ for all $i, j \in \omega$. Let us prove that
$$
    g\colon \langle X, \gamma \rangle \to Z \text{ is continuous.}
$$

Let $U \in \tau_Z$ and $x \in g^{-1}[U]$. Now we consider two cases.

1. $x \in X \setminus D_X$. From (\ref{X_minus_DX_Di}) it follows that there exists $i \in \omega$ such that $x \in D_i$. Let $j \in \omega$ be such that $U_{g(x)}^j \subseteq U$, then $W_i(x) \cap g^{-1}[U_{g(x)}^j] \subseteq g^{-1}[U]$ and $W_i(x) \cap g^{-1}[U_{g(x)}^j] \in \mathcal{W}_i^j \subseteq \gamma$.

(2) $x \in D_X$. Then there exists $i \in \omega$ and $F \in \mathcal{D}_i$ such that $x \in F$. Let $j \in \omega$ be such that $U_{g[F]}^j \subseteq U$, then $V_i(F) \cap g^{-1}[U_{g[F]}^j] \subseteq g^{-1}[U]$ and $V_i(x) \cap g^{-1}[U_{g[F]}^j] \in \mathcal{V}_i^j \subseteq \gamma$.

Now from (\ref{sigma_connected}) it follows that $\langle X, \gamma \rangle$ is connected, and since $g\colon X \to Z$ is open it follows that $g\colon \langle X, \gamma \rangle \to Z$ is open too.
\end{proof}

\begin{corr}
    A $T_1$ topological space is connected and first-countable if and only if it is a continuous open image of a connected metrizable space. 
\end{corr}

\begin{ques}
    Let $X$ be a connected Lashnev space. Is it true that $X$ is a continuous closed image of a metrizable connected space?
\end{ques}

{\bf Acknowledgement} The work was performed as part of research conducted in the Ural Mathematical Center with the financial support of the Ministry of Science and Higher Education of the Russian Federation (Agreement number 075-02-2024-1377).

\end{document}